\documentclass[11pt]{amsart}%
\usepackage{hyperref}
\usepackage{amsmath}%
\usepackage{amsfonts}%
\usepackage{amssymb}%
\usepackage{amsthm}
\usepackage{mathrsfs}
\usepackage{comment}
\usepackage{todonotes}
\usepackage{bbm}
\usepackage{amscd}
\usepackage{youngtab}
\usepackage{tikz}
\usepackage{bm}
\usepackage{enumitem}
\usetikzlibrary{arrows}
\usetikzlibrary{matrix}

\newcommand{\ph}{\varphi}

\newcommand{\C}{\mathbb{C}}
\newcommand{\Z}{\mathbb{Z}}

\newcommand{\mf}{\mathfrak}

\renewcommand{\ker}{\text{ker}}

\newcommand{\im}{\text{im}}

\newcommand{\rank}{\text{rank}}

\DeclareMathOperator{\Res}{Res}
\DeclareMathOperator{\Ind}{Ind}
\newcommand{\Irr}{\text{Irr}}

\renewcommand{\hat}{\widehat}

\newtheorem{theorem}{Theorem}[section]
\newtheorem{def-prop}[theorem]{Definition-Proposition}
\newtheorem{prop}[theorem]{Proposition}
\newtheorem{conjecture}[theorem]{Conjecture}
\newtheorem{lemma}[theorem]{Lemma}

\theoremstyle{definition}
\newtheorem{definition}[theorem]{Definition}
\newtheorem{example}[theorem]{Example}

\theoremstyle{remark}
\newtheorem*{remark}{Remark}

\begin{document}

\title{Dual graded graphs and Bratteli diagrams of towers of groups}
\author{Christian Gaetz}
\address{Department of Mathematics, Massachusetts Institute of Technology, Cambridge, MA.}
\email{\href{mailto:gaetz@mit.edu}{gaetz@mit.edu}} 
\date{\today}

\begin{abstract}
An $r$-dual tower of groups is a nested sequence of finite groups, like the symmetric groups, whose Bratteli diagram forms an $r$-dual graded graph.  Miller and Reiner introduced a special case of these towers in order to study the Smith forms of the up and down maps in a differential poset.  Agarwal and the author have also used these towers to compute critical groups of representations of groups appearing in the tower.  In this paper the author proves that when $r=1$ or $r$ is prime, wreath products of a fixed group with the symmetric groups are the only $r$-dual tower of groups, and conjecture that this is the case for general values of $r$.  This implies that these wreath products are the only groups for which one can define an analog of the Robinson-Schensted bijection in terms of a growth rule in a dual graded graph.  
\end{abstract}

\maketitle

\section{Introduction}


\subsection{Differential posets and dual graded graphs}
\label{sec:dp-dgg-defs}
Differential posets are a class of partially ordered sets introduced by Stanley \cite{Stanley1988} which generalize many of the enumerative and combinatorial properties of Young's lattice $Y$, the poset of integer partitions ordered by inclusion of Young diagrams.  The reader should see \cite{Stanley2012} for basic definitions and conventions for posets in what follows.  Dual graded graphs are a generalization of differential posets developed independently by Fomin \cite{Fomin1994, Fomin1995}.

A \textit{graded graph} is an undirected multigraph $P$ together with a rank function $\rho: P \to \Z_{\geq 0}$ such that $P$ has a unique element $\hat{0}$ of rank zero, all ranks $P_n=\rho^{-1}(\{n\})$ are finite, and such that all edges are between consecutive ranks: if $(x,y)$ is an edge, then $|\rho(x)-\rho(y)|=1$; we denote the multiplicity of this edge by $m(x,y)$.  In analogy with Hasse diagrams of partially ordered sets, we write $x \leq y$ if $\rho(x) \leq \rho(y)$ and there is a path from $x$ to $y$ in $P$ taking only upward steps.  If $\rho(y)=\rho(x)+1$ and $x < y$, we write $x \lessdot y$ and say that $y$ \textit{covers} $x$.  For $a,b \in \Z_{\geq 0}$, we let $P_{[a,b]}$ denote the induced subgraph on the elements of ranks $a,a+1,...,b$.

\begin{definition}
\label{def:UD-condition}
Let $r$ be a postitive integer, let $P$ be a graded graph, and let $\C P$ denote the complex vector space with basis $P$; then $P$ is an \textit{$r$-dual graded graph} if $DU-UD=rI$ where the linear operators $U,D : \C P \to \C P$ are defined by
\begin{align*}
Ux &= \sum_{x \lessdot y} m(x,y) y \\
Dy &= \sum_{x \lessdot y} m(x,y) x.
\end{align*}
If in addition all edges multiplicities $m(x,y)$ are 0 or 1, then $P$ is (strictly speaking, the Hasse diagram of) an \textit{$r$-differential poset}.
\end{definition}

\begin{remark}
What we have defined here are usually called \textit{self-dual} graded graphs.  In the context of towers of groups, self-duality is implied by Frobenius reciprocity, so we do not address the more general definition here.
\end{remark}

\begin{center}
\begin{figure}
\label{fig:youngs-lattice}
\includegraphics[scale=0.2]{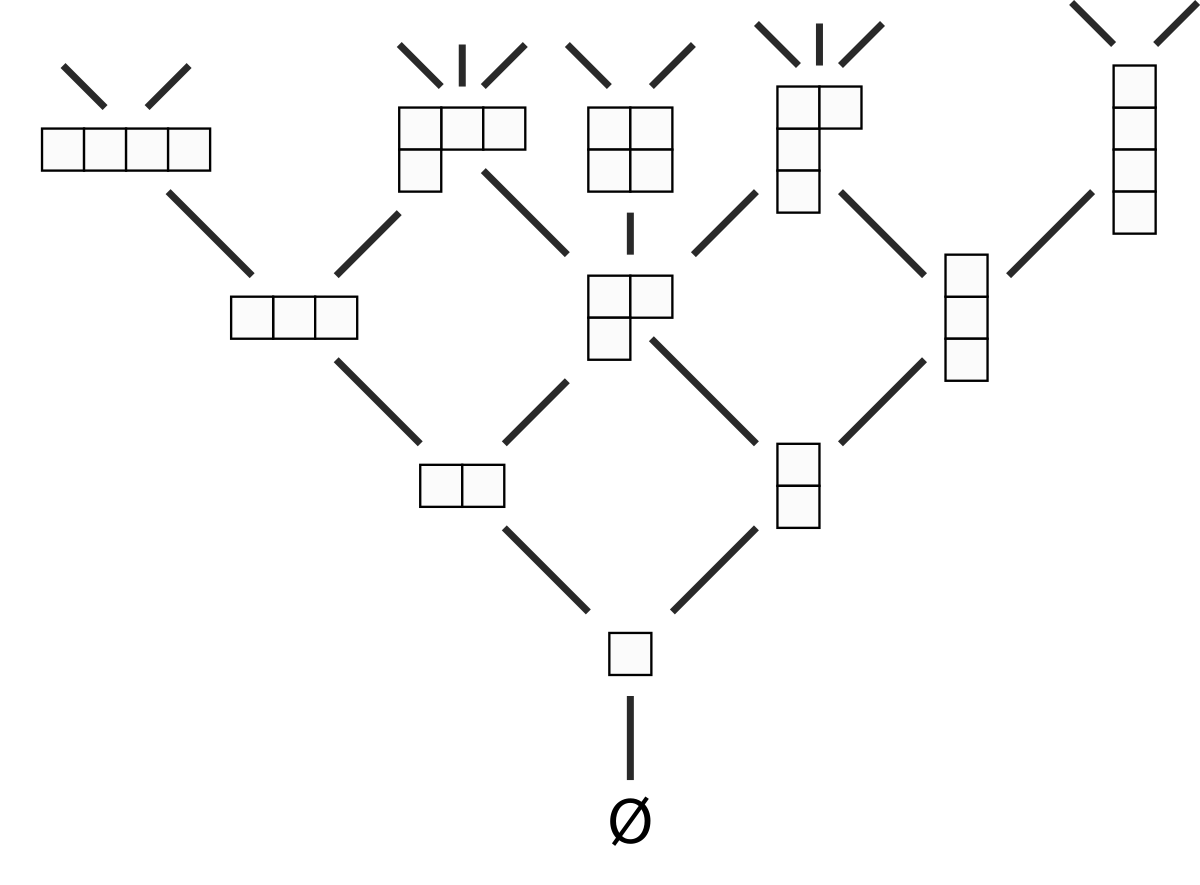}
\caption{Young's lattice $Y$ is a 1-dual graded graph and a 1-differential poset.}
\end{figure}
\end{center}

Many of the combinatorial properties of Young's lattice $Y$ are shared by all dual graded graphs.  First, define a pairing $\langle , \rangle : \C P \times \C P \to \C$ by requiring that $\langle x , y \rangle = \delta_{xy}$ for $x,y \in P$ and extending by bilinearity; let $e(x)=\langle U^n \hat{0}, x \rangle$ where $x \in P_n$.  It is easy to see that $e(x)$ counts the number of paths from $\hat{0}$ to $x$ in $P$, allowing only upward steps.  Then we have

\begin{prop}[\cite{Fomin1994}]
\label{prop:sum-of-squares}
\[
\sum_{x \in P_n} e(x)^2 = r^n n!
\]
\end{prop}

When $P=Y$, it is easy to see that $e(\lambda)$ is the number $f_{\lambda}$ of standard Young tableaux of shape $\lambda$, and so Proposition \ref{prop:sum-of-squares} reduces to the well known fact that $\sum_{\lambda \in Y_n} f_{\lambda}^2 = n!$, which follows from the Robinson-Schensted correspondence.  

\subsection{Towers of groups}
\label{sec:towers}
A \textit{tower of groups} $\mf{G}$ is an infinite nested sequence of finite groups $\mf{G}: \{e\}=G_0 \subset G_1 \subset G_2 \subset \cdots$ (see Section \ref{sec:strong-towers} where this notion is distinguished from the more restrictive sense in which Bergeron, Lam, and Li \cite{Bergeron2012} use the term \textit{tower of algebras}).

For a finite group $G$, we let $R(G)$ denote the representation ring of complex linear combinations of complex  $G$-representations subject to the relations $[V \oplus W]=[V]+[W]$ and $[V \otimes W] = [V] \cdot [W]$.  As a vector space, $R(G)$ has a distinguished basis consisting of the classes of irreducible representations $[V_{\lambda}]$ for $V_{\lambda} \in \Irr(G)$.  For $\mf{G}$ a tower of groups, we let
\[
R(\mf{G})=\bigoplus_{i=0}^{\infty} R(G_i).
\] 
Although each $R(G_i)$ has a ring structure, we regard $R(\mf{G})$ only as a complex vector space (see Section \ref{sec:strong-towers}), together with an inner product $\langle V_{\lambda}, V_{\mu} \rangle = \delta_{\lambda \mu}$ defined so that the distinguished basis of irreducibles is orthonormal; this inner product coincides with the usual inner product of characters when restricted to a subspace $R(G_i)$.

The space $R(\mf{G})$ has a natural pair of adjoint linear operators $\Ind, \Res: R(\mf{G}) \to R(\mf{G})$ defined by
\begin{align*}
\Ind([V_{\lambda}])&=[\Ind_{G_i}^{G_{i+1}} V_{\lambda}]  \\
\Res([V_{\lambda}]) &=[\Res^{G_i}_{G_{i-1}} V_{\lambda}] 
\end{align*}
where $V_{\lambda} \in \Irr(G_i)$, and extended by linearity.  We define $\Res([V])=0$ for $[V] \in R(G_0) \cong \C$.  We now define the main object of study.

\begin{definition}[\cite{Miller2009}] \label{def:dtg}
A tower of groups $\mf{G}$ is an $r$-\textit{dual tower of groups} if the linear operators $\Ind, \Res: R(\mf{G}) \to R(\mf{G})$ satisfy the relation 
\begin{equation} \label{eq:DTG}
\Res \Ind - \Ind \Res = r I.
\end{equation}
That is, $\mf{G}$ is an $r$-dual tower of groups if and only if the Bratteli diagram of $\mf{G}$ is an $r$-dual graded graph $P$.  In this case the operators $\Ind, \Res: R(\mf{G}) \to R(\mf{G})$ correspond to $U,D: \C P \to \C P$ and the inner products on $\C P$ and $R(\mf{G})$ coincide; we write $P=P(\mf{G})$.  We also abuse terminology by saying $V_{\lambda}$ covers  $V_{\mu}$ if $\lambda$ covers $\mu$ in $P$.

If in addition the branching rules between $G_i$ and $G_{i+1}$ are multiplicity-free for all $i$ then we say that $\mf{G}$ is an $r$-\textit{differential tower of groups}.  In this case the Bratteli diagram forms an $r$-differential poset $P=P(\mf{G})$.
\end{definition}

Miller and Reiner introduced differential towers of groups in order to study the Smith forms of the up and down maps in the associated differential poset \cite{Miller2009}.  Agarwal and the author have also used these towers to compute critical groups of representations of groups appearing in the tower \cite{Agarwal2017, Gaetz2016a}. 

\begin{example} \label{ex:tower-of-symmetric-groups}
It is well known (see for example \cite{James1981}) that irreducible representations of the symmetric group $S_n$ are indexed by partitions $\lambda$ of $n$, and that 
\[
\Res^{S_n}_{S_{n-1}} V_{\lambda} = \bigoplus_{\mu} V_{\mu}
\]
where the direct sum is over all partitions $\mu$ of $n-1$ whose Young diagrams are obtained from that of $\lambda$ by removing a single box.  That is, $\mf{S}: \{e\} \subset S_1 \subset S_2 \subset \cdots$ is a 1-differential tower of groups, with $P(\mf{S})=Y$.

This correspondence can be extended to show that for any abelian group $A$ of order $r$ the tower 
\[
A \wr \mf{S}: \{e\} \subset A \subset A \wr S_2 \subset A \wr S_3 \subset \cdots
\]
is an $r$-differential tower of groups with $P(A \wr \mf{S}) \cong Y^r$ \cite{Okada1990}.  When $A$ is cyclic, $A \wr \mf{S}$ is a complex reflection group \cite{Shephard1954}; the trivial representation corresponds to the tuple $((n),\emptyset,...,\emptyset)$ of partitions, and the reflection representations are of the form $((n-1),\emptyset,...,\emptyset,(1),\emptyset,...,\emptyset)$ when $r>1$ or $((n-1,1))$ when $r=1$.  If $\boldsymbol{\lambda}=(\lambda^1,...,\lambda^r)$ is a tuple of partitions with total size $n$, then the dimension of the corresponding representation $V_{\boldsymbol{\lambda}}$ of $A \wr S_n$ is:
\[
\binom{n}{|\lambda^1|,...,|\lambda^r|} \prod_{i=1}^r f_{\lambda^i}
\]
where $f_{\mu}$ is the dimension of the irreducible $S_n$ representation indexed by $\mu$.

More generally, if $H$ is any finite group of order $r$ with $k$-conjugacy classes and irreducible representations of dimensions $d_1,...,d_k$, then 
\[
H \wr \mf{S}: \{e\} \subset H \subset H \wr S_2 \subset H \wr S_3 \subset \cdots
\]
is an $r$-dual tower of groups with $P(H \wr S_n) \cong (d_1Y) \times (d_2Y) \times \cdots \times (d_kY)$ where $dY$ denotes the $d^2$-dual graded graph whose underlying simple graph is $Y$, but where all edge multiplicities are multiplied by $d$ (see \cite{Okada1990}, Theorem 4.1).

Theorem \ref{thm:main-theorem} below shows that this is in fact the \textit{only} $r$-dual tower of groups when $r=1$ or $r$ is prime.  In the rest of the paper, we use the convention that $A \wr S_0 = \{e\}$ while $A \wr S_1 = A$.

In \cite{Okada1990}, Okada also gives an explicit bijection which generalizes the \\ Robinson-Schensted correspondence to a bijection between elements of the groups $H \wr S_n$ and pairs of paths in $Y^r$.  Fomin's theory of dual graded graphs \cite{Fomin1994, Fomin1995} later showed that this bijection was a special case of one which holds for any dual graded graph, defined in terms of growth rules.  The main results of this paper, Theorem \ref{thm:main-theorem} and Conjecture \ref{conj:main-conjecture} below, together assert that the groups $H \wr S_n$ are the \textit{only} groups for which one can define an analog of the Robinson-Schensted correspondence via growth rules. 
\end{example}

\begin{theorem} \label{thm:main-theorem}
Let $\mf{G}: \{e\}=G_0 \subset G_1 \subset G_2 \subset \cdots$ be an $r$-dual tower of groups where $r$ is one or prime.  Then $P(\mf{G}) \cong Y^r$ and $G_n \cong (\Z/r \Z) \wr S_n$ for all $n \geq 0$.
\end{theorem}

\begin{conjecture} \label{conj:main-conjecture} Let $\mf{G}: \{e\}=G_0 \subset G_1 \subset G_2 \subset \cdots$ be a tower of groups and let $r \in \Z_{>0}$.
\begin{itemize}
\item[(a)] If $\mf{G}$ is $r$-differential, then $P(\mf{G}) \cong Y^r$ and there exists an abelian group $A$ of order $r$, not depending on $n$, such that $G_n \cong A \wr S_n$ for all $n \geq 0$. 
\item[(b)] If $\mf{G}$ is an $r$-dual tower of groups then $P(\mf{G}) \cong (d_1Y) \times \cdots \times (d_kY)$ for some $d_1,...,d_k$ and there exists a group $H$ of order $r$, not depending on $n$, such that $G_n \cong H \wr S_n$ for all $n \geq 0$. 
\end{itemize}
\end{conjecture}

Clearly part (b) of the conjecture implies part (a).

\begin{remark}
The general construction outlined in \cite{Goodman1989} shows that any graded graph can be realized as the Bratteli diagram of some sequence $A_0 \subset A_1 \subset A_2 \subset \cdots$ of complex semisimple algebras.  In particular, for all $r$ each of the (uncountably many) $r$-differential posets and $r$-self-dual graded graphs can be realized in this way; thus the implication in Theorem \ref{thm:main-theorem} that $Y^r$ is the only differential poset or dual graded graph which can arise from a sequence of \textit{group algebras} may initially be surprising.
\end{remark}

\subsection{Relation to Bergeron, Lam, and Li's towers of algebras}
\label{sec:strong-towers}

In \cite{Bergeron2012}, Bergeron, Lam, and Li study what they  call towers of algebras; I will call their notion a \textit{strong tower of algebras} in order to avoid confusion.  Strong towers of algebras are certain algebras $A=\bigoplus_{i=0}^n A_i$ whose summands $A_i$ are algebras in their own right which are subject to several additional conditions; the corresponding branching rules satisfy the same differential relation as in (\ref{eq:DTG}).  A pair of Grothendieck groups $(G(A),K(A))$ bears a similar relationship to $A$ as the vector space $R(\mf{G})$ does to a tower of groups $\mf{G}$.  The additional conditions required of a strong tower of algebras are restrictive enough to guarantee that $G(A)$ and $K(A)$ in fact have the structure of dual combinatorial Hopf algebras.  If the $A_i$ are complex semisimple algebras one obtains that $G(A)=K(A)$ is a positive self-dual Hopf algebra.  By Zelevinsky's classification of such Hopf algebras \cite{Zelevinsky1981}, this forces $G(A)$ to be isomorphic to a tensor product of copies of the Hopf algebra $\Lambda$ of symmetric functions, and therefore forces the corresponding Bratteli diagram for $A$ to be isomorphic to a product of copies of Young's lattice.

In the case where $A_i = \C[G_i]$ are group algebras for $i=0,1,...$, the additional conditions required of a strong tower of algebras imply in particular that for all $n \geq k \geq 0$ the group $G_k \times G_{n-k}$ is isomorphic to a subgroup of $G_n$.  This condition is clearly not true for general towers of groups, and it is therefore a notable feature of Theorem \ref{thm:main-theorem} and Conjecture \ref{conj:main-conjecture} that this condition, and the resulting Hopf structure on $R(\mf{G})$, emerge from the local condition (\ref{eq:DTG}) which only relates the branching rules of consecutive pairs of groups in the tower.

\section{Proof of Theorem \ref{thm:main-theorem}} \label{sec:proof}

Throughout Section \ref{sec:proof}, let $\mf{G}: \{e\}=G_0 \subset G_1 \subset G_2 \subset \cdots$ be an $r$-dual tower of groups where $r$ is one or prime.  The proof of Theorem \ref{thm:main-theorem} will proceed by showing inductively that the groups $G_n$ are complex reflection groups and applying the known classification \cite{Shephard1954} of these groups.  We first prove the following useful fact:   

\begin{prop} \label{prop:size-of-Gn}
For all $n \geq 0$, the group $G_n$ has order $r^nn!$.
\end{prop}
\begin{proof}
First, note that $e(\lambda)=\dim(V_{\lambda})$, since both satisfy the same recurrence (as the sum of the values for $\mu \lessdot \lambda$, with multiplicity) and the initial condition $e(\hat{0})=\dim(\mathbbm{1}_{G_0})=1$.  Thus by Proposition \ref{prop:sum-of-squares} and the standard fact that the sum of the squares of the dimensions of the irreducibles is the order of the group we have:
\[
|G_n| = \sum_{V_{\lambda} \in \Irr(G_n)} \dim(V_{\lambda})^2 = \sum_{\lambda \in P_n} e(\lambda)^2 = r^nn!.
\]
\end{proof}

Before proceeding with the proof, we state Clifford's Theorem, which will be used several times in what follows.

\begin{theorem}[Clifford's Theorem]
Let $G$ be a finite group, $N$ a normal subgroup and $\ph:G \to GL(V)$ an irreducible representation of $G$.  Then the irreducible factors appearing in $\Res^{G}_N V$ are in a single orbit under the action of $G$ given by $\psi^{(g)}(n)=\psi(gng^{-1})$.

In particular, all irreducible factors of $\Res^G_N V$ are of the same dimension, and if $\mathbbm{1}_N$ appears as a factor, then $\Res^G_N V$ is a direct sum of trivial representations.  
\end{theorem}

\subsection{Base case} \label{sec:base-case}

By assumption we have that $G_0 = \{e\}$ is trivial, and since $r$ is 1 or prime, we know by Proposition \ref{prop:size-of-Gn} that $G_1 \cong \Z/r \Z$.  In this section we will show:

\begin{prop} \label{prop:base-case}
When $r$ is one or prime we must have $G_2 \cong (\Z/ r \Z) \wr S_2$ and $P_{[0,2]} \cong (Y^r)_{[0,2]}$.
\end{prop}
\begin{proof}
When $r=1$ the claim is clear, since $|G_2|=2$, so assume $r$ is prime.  If $r \neq 2$, then the Sylow $r$-subgroup of $G_2$ has order $r^2$ and so must be abelian; since it has index two, it is also normal in $G_2$.  It is a standard fact (see \cite{Serre1977}) that dimensions of irreducible representations must divide the index of an abelian normal subgroup.  Thus all irreducible representations of $G_2$ have dimension 1 or 2.  In order to satisfy the relation (\ref{eq:DTG}), there must be $\binom{r}{2}$ 2-dimensional irreducibles of $G_2$ which when restricted to $G_1$ give each of the possible pairs of characters, and there must be $2r$ linear characters, two of which restrict to each character of $G_1$.  It is a straightforward exercise (see, for example, \cite{Dummit2004} p. 185) to check that there are three nonabelian groups $G$ of order $2r^2$, all of which are semidirect products of $H=\langle h \rangle  \cong \Z/2\Z$ and the $r$-Sylow subgroup $N$.  If $N$ the cyclic group $C_{r^2}$, then $h$ acts on $N$ by inversion; if $N$ is $C_r \times C_r$, then $h$ acts either by inversion of a single factor, or inversion of both.  In the first and last cases one calculates that the commutator subgroup $[G,G]$ is of order $r^2$, thus $G$ has only two linear characters and cannot be $G_2$; the only remaining possibility is $G_2 = (\Z/r \Z) \wr S_2$.

When $r=2$, there is only one 2-dual graded graph up to rank 2, namely $(Y^2)_{[0,2]}$.  These branching rules imply that $G_1$ is not normal in $G_2$, since the 2-dimensional irreducible restricts to the sum of the trivial representation and a non-trivial summand, violating Clifford's theorem. The only nonabelian groups of order 8 are the dihedral group $D_4 \cong (\Z/2 \Z) \wr S_2$ and the quaternion group $Q_8$.  However $Q_8$ is well-known to be a Hamiltonian group, contradicting the fact that $G_1$ is not normal.  Thus we must have $G_2 \cong (\Z/2\Z) \wr S_2$ as desired.
\end{proof}

\subsection{Facts about dual graded graphs} \label{sec:dgg-facts}

\begin{lemma} \label{lem:dgg-facts} Let $P$ be an $r$-dual graded graph such that $P_{[0,m]}$ has no multiple edges and let $x \neq y \in P_m$.  Then
\begin{itemize}
\item[(a)] If $z,z' \in P_{m+1}$ both cover both $x$ and $y$, then $z=z'$.
\item[(b)] If $z \in P_{m+1}$ covers both $x$ and $y$, then there is some $w \in P_{m-1}$ which is covered by both $x$ and $y$ and $m(x,z)=m(y,z)=m(w,x)=m(w,y)=1$.
\end{itemize}
\end{lemma}
\begin{proof} \text{}
\begin{itemize}
\item[(a)] Suppose $z \neq z'$, then the coefficient of $y$ in $DUx$ is at least 2.  Therefore the coefficient of $y$ in $UDx$ must also be at least 2.  Since $P_{[0,m]}$ has no multiple edges, there must be some $w \neq w' \in P_{m-1}$ such that $x,y$ both cover $w$ and $w'$.  Repeating this argument, we arrive at a contradiction, since $P_0$ has a single element.  Thus $z=z'$.
\item[(b)] Since $z$ covers $x$ and $y$, we know $y$ appears in $DUx$, and so it must also appear in $UDx$; that is, there must be some $w \in P_{m-1}$ covered by $x$ and $y$, and by part (a) this $w$ is unique.  There are no multiple edges in $P_{[0,m]}$, so $y$ appears with coefficient 1 in $UDx$, thus it must also have coefficient 1 in $DUx$ which forces $m(x,z)=m(y,z)=m(w,x)=m(w,y)=1$.
\end{itemize}
\end{proof}

\subsection{Inductive step} \label{sec:inductive-step} 

By a \textit{partial $r$-dual tower of groups} we mean a finite sequence of finite groups $\{e\}=G_0 \subset G_1 \subset \cdots \subset G_{m+1}$ such that for $i=1,...,m$ we have
\[
\Res^{G_{i+1}}_{G_{i}} \Ind_{G_{i}}^{G_{i+1}} - \Ind_{G_{i-1}}^{G_{i}} \Res_{G_{i-1}}^{G_{i}} = r I.
\]
and we similarly define a partial $r$-dual graded graph.

\begin{prop} \label{prop:inductive-step}
For $m \geq 2$, let $\{e\}=G_0 \subset G_1 \subset \cdots \subset G_{m+1}$ be a partial $r$-dual tower of groups with corresponding partial dual graded graph $P$. Suppose that for $n \leq m$ we have $G_n \cong (\Z/r \Z) \wr S_n$ and that $P_{[0,m]} \cong (Y^r)_{[0,m]}$.  Then $G_{m+1} \cong (\Z/r \Z) \wr S_{m+1}$ and $P_{[0,m+1]} \cong (Y^r)_{[0,m+1]}$.
\end{prop}
\begin{proof}
The goal is to show that $G_{m+1}$ must be an irreducible complex reflection group of rank $m+1$ by identifying a faithful irreducible representation $W$ of dimension $m+1$ and a generating set for $G_{m+1}$ such that all elements of the generating set act in $W$ as complex reflections.

Since $P_{[0,m]} \cong (Y^r)_{[0,m]}$ we know that $\Ind_{G_{m-1}}^{G_m} \mathbbm{1}_{G_{m-1}}$ decomposes into distinct irreducibles as $\mathbbm{1}_{G_m} \oplus \bigoplus_{i=1}^r U_i$ where $U_1$ has dimension $m-1$ and $U_i$ has dimension $m$ for $i \geq 2$.  Define distinct irreducible representations $W_k$ by $\Ind_{G_m}^{G_{m+1}} \mathbbm{1}_{G_m} \cong \mathbbm{1}_{G_{m+1}} \oplus \bigoplus_{k} W_k^{\oplus c_k}$ with $c_k>0$.  Let $N$ be the smallest normal subgroup of $G_{m+1}$ which contains $G_m$.  By Lemma \ref{lem:dgg-facts}, there must be some $W_k$ whose restriction to $G_m$ contains both $\mathbbm{1}_{G_m}$ and $U_1$, since both of these representations cover $\mathbbm{1}_{G_{m-1}}$.  Then by Clifford's Theorem, $G_m$ is not normal in $G_{m+1}$ since $\dim(\mathbbm{1}_{G_m}) \neq \dim(U_1)$.  Thus $N$ properly contains $G_m$.  Let the $V_j$ be the distinct nontrivial irreducible representations of $N$ appearing in $\Res^{G_{m+1}}_N W_k$ for some $k$.

We now collect some facts about the representations $U_i, V_j$ and $W_k$.  In order to avoid naming homomorphisms $\ph: G \to GL(V)$ corresponding to all representations $V$, we abuse notation by writing $\ker(V)$ for $\ker(\ph)$.

\begin{lemma} \label{lem:UVW-facts} \text{}
\begin{enumerate}
\item[(a)] Either $\Res^{G_{m+1}}_N W_k = \mathbbm{1}_N^{\oplus c}$ or $\Res^{G_{m+1}}_{G_m} W_k$ is a multiplicity-free direct sum of $\mathbbm{1}_{G_m}$ and some of the $U_i$.  In the latter case, $\Res^{G_{m+1}}_N W_k$ contains a unique $V_j$ such that $\Res^N_{G_m} V_j$ contains $\mathbbm{1}_{G_m}$, and $V_j$ does so with multiplicity one.
\item[(b)] $\Res^N_{G_m} V_j$ must contain some $U_i$.
\item[(c)] No distinct $W_k, W_{k'}$ may contain the same $V_j$ in their restriction to $N$.
\item[(d)] No distinct $V_j, V_{j'}$ may contain the same $U_i$ in their restriction to $G_m$.
\end{enumerate}
\end{lemma}
\begin{proof}[Proof of lemma \ref{lem:UVW-facts}] \text{}
\begin{enumerate} 
\item[(a)] Since $N$ is normal, by Clifford's Theorem either all irreducible factors in $\Res^{G_{m+1}}_N W_k$ are trivial, or none are; in the first case we are done, so assume we are in the second case.  By definition, $\Res^{G_{m+1}}_{G_m} W_k$ must contain $\mathbbm{1}_{G_m}$; however it must also contain some nontrivial factor, otherwise $G_m \subseteq \ker(W_k)$ but $N \not \subseteq \ker(W_k)$, contradicting the fact that $N$ is the smallest normal subgroup containing $G_m$.  Thus $\Res^{G_{m+1}}_{G_m} W_k$ contains $\mathbbm{1}_{G_m}$ and some nontrivial irreducible $U$.  By Lemma \ref{lem:dgg-facts}, $U$ must cover $\mathbbm{1}_{G_{m-1}}$, so $U$ is one of the $U_i$ and both $\mathbbm{1}_{G_m}$ and $U_i$ appear with multiplicity one in $\Res^{G_{m+1}}_{G_m} W_k$.  The last statement now follows immediately.

\item[(b)] Suppose $V_j$ does not contain any $U_i$ in its restriction; pick $W_k$ which contains $V_j$ in its restriction.  Then, to satisfy the multiplicity freeness condition in part (a) we must have $\Res^{N}_{G_m} V_j = \mathbbm{1}_{G_m}$, so $\dim(V_j)=1$.  By Clifford's theorem $W_k$ only contains 1-dimensional representations in its restriction to $N$.  But any other possible 1-dimensional representation appearing in this restriction itself restricts to $\mathbbm{1}_{G_m}$, so in order to satisfy multiplicity freeness, it must be that $\dim(W_k)=1$ and $\Res^{G_{m+1}}_N W_k=V_j$.  But then $\ker(W_k)$ is a normal subgroup of $G_{m+1}$ which contains $G_m$ but which does not contain $N$, contradicting the minimality of $N$.

\item[(c)] Suppose $\Res^{G_{m+1}}_N W_k$ and $\Res^{G_{m+1}}_N W_{k'}$ both contain $V_j$.  Since, by part (b), $\Res^N_{G_m} V_j$ must contain some $U_i$ we see that $W_k, W_{k'}$ both cover both $\mathbbm{1}_{G_m}$ and $U_i$, thus by Lemma \ref{lem:dgg-facts}, $W_k=W_{k'}$.

\item[(d)] Suppose $\Res^{N}_{G_m} V_j$ and $\Res^{N}_{G_m} V_{j'}$ both contain $U_i$.  If $V_j,V_{j'}$ are both contained in the restriction of the same $W_k$, then this violates multiplicity freeness from part (a).  Otherwise, $V_j, V_{j'}$ are contained in the restrictions of $W_k \neq W_{k'}$ respectively and then both $W_k, W_{k'}$ cover both $\mathbbm{1}_{G_m}$ and $U_i$, contradicting $  $Lemma \ref{lem:dgg-facts}.

\end{enumerate}
\end{proof}

\begin{center}
\begin{figure} \label{fig:two-UVW-possibilities}
\begin{tikzpicture} 
\node(1111) at (-3,3) {$\mathbbm{1}_{G_{m+1}}$};
\node(111) at (-2,2) {$\mathbbm{1}_{N}$};
\node(11) at(-1,1) {$\mathbbm{1}_{G_{m}}$};
\node(1) at (0,0) {$\mathbbm{1}_{G_{m-1}}$};
\node(u1) at (0,1) {$U_1$};
\node(u2) at (1,1) {$U_2$};
\node(u3) at (2,1) {$U_3$};
\node(v1) at (-1,2) {$V_1$};
\node(v2) at (0,2) {$V_2$};
\node(v3) at (1,2) {$V_3$};
\node(w1) at (-2,3) {$W_1$};
\node(w2) at (-1,3) {$W_2$};
\node(w3) at (0,3) {$W_3$};

\draw[thick, shorten <=-4pt, shorten >=-4pt] (111)--(1111);
\draw[thick, shorten <=-4pt, shorten >=-4pt] (11)--(111);
\draw[thick, shorten <=-4pt, shorten >=-4pt] (1)--(11);

\draw[thick, shorten <=-5pt, shorten >=-2pt] (1)--(u1);
\draw[thick, shorten <=-5pt, shorten >=-2pt] (1)--(u2);
\draw[thick, shorten <=-5pt, shorten >=-2pt] (1)--(u3);

\draw[thick, shorten <=-5pt, shorten >=-2pt] (11)--(v1);
\draw[thick, shorten <=-5pt, shorten >=-2pt] (11)--(v2);
\draw[thick, shorten <=-5pt, shorten >=-2pt] (11)--(v3);

\draw[thick, shorten <=-2pt, shorten >=-2pt] (u1)--(v1);
\draw[thick, shorten <=-2pt, shorten >=-2pt] (u2)--(v2);
\draw[thick, shorten <=-2pt, shorten >=-2pt] (u3)--(v3);

\draw[thick, shorten <=-2pt, shorten >=-2pt] (v1)--(w1);
\draw[thick, shorten <=-2pt, shorten >=-2pt] (v2)--(w2);
\draw[thick, shorten <=-2pt, shorten >=-2pt] (v3)--(w3);
\end{tikzpicture}
\begin{tikzpicture} 
\node(1111) at (-3,3) {$\mathbbm{1}_{G_{m+1}}$};
\node(111) at (-2,2) {$\mathbbm{1}_{N}$};
\node(11) at(-1,1) {$\mathbbm{1}_{G_{m}}$};
\node(1) at (0,0) {$\mathbbm{1}_{G_{m-1}}$};
\node(u1) at (0,1) {$U_1$};
\node(u2) at (1,1) {$U_2$};
\node(u3) at (2,1) {$U_3$};
\node(v1) at (-1,2) {$V_1$};
\node(v2) at (0,2) {$V_2$};
\node(v3) at (1,2) {$V_3$};
\node(w1) at (-2,3) {$W_1$};
\node(w2) at (-1,3) {$W_2$};
\node(w3) at (0,3) {$W_3$};

\draw[thick, shorten <=-4pt, shorten >=-4pt] (111)--(1111);
\draw[thick, shorten <=-4pt, shorten >=-4pt] (11)--(111);
\draw[thick, shorten <=-4pt, shorten >=-4pt] (1)--(11);

\draw[thick, shorten <=-5pt, shorten >=-2pt] (1)--(u1);
\draw[thick, shorten <=-5pt, shorten >=-2pt] (1)--(u2);
\draw[thick, shorten <=-5pt, shorten >=-2pt] (1)--(u3);

\draw[thick, shorten <=-5pt, shorten >=-2pt] (11)--(v1);

\draw[thick, shorten <=-2pt, shorten >=-2pt] (u1)--(v1);
\draw[thick, shorten <=-2pt, shorten >=-2pt] (u2)--(v2);
\draw[thick, shorten <=-2pt, shorten >=-2pt] (u3)--(v3);

\draw[thick, shorten <=-2pt, shorten >=-2pt] (v1)--(w3);
\draw[thick, shorten <=-2pt, shorten >=-2pt] (v2)--(w3);
\draw[thick, shorten <=-2pt, shorten >=-2pt] (v3)--(w3);
\draw[thick, shorten <=-5pt, shorten >=-2pt] (111)--(w1);
\draw[thick, shorten <=-5pt, shorten >=-2pt] (111)--(w2);
\end{tikzpicture}
\caption{The two possibilities for the branching rules between the $U_i,V_j,W_k$ when $r=3$.  On the left is the case $r'=1$ and on the right the case $r'=r=3$.}
\end{figure}
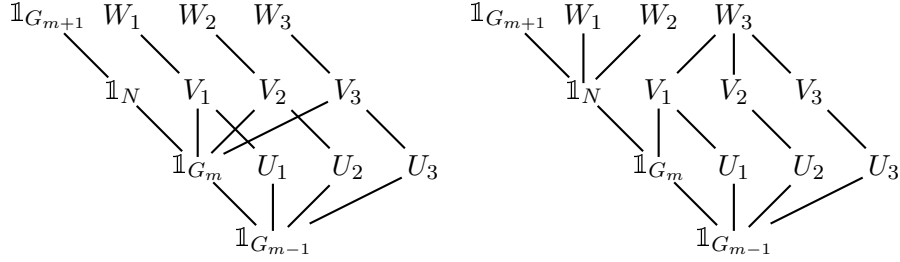
\end{center}

We now return to the proof of Proposition \ref{prop:inductive-step}.  Define $r'=[G_{m+1}:N]$.  Since representations of $G_{m+1}/N$ are in dimension-preserving bijection with representations of $G_{m+1}$ which restrict to a multiple of the trivial representation on $N$, we see that $r'=1+ \sum a_k^2$ where $a_k$ is the multiplicity of $\mathbbm{1}_N$ in $\Res^{G_{m+1}}_N W_k$.  Let $b_k>0$ denote the multiplicity of $\mathbbm{1}_{G_m}$ in $\Res^{G_{m+1}}_{G_m} W_k$; clearly we have $a_k \leq b_k$, and $a_k=b_k$ if $a_k>0$.  By the definition of an $r$-dual tower of groups, we know $\sum_k b_k^2 = r$.  Furthermore, at least one $W_{k'}$ must contain some $U_i$ in its restriction, since $\mathbbm{1}_{G_m}$ and $U_i$ must have an upper bound in rank $m+1$, and thus $a_{k'}=0$ by Lemma \ref{lem:UVW-facts}.  Therefore $1 \leq r' \leq r$.  

Now, we have
\[
\dim \Ind_{G_m}^{N} \mathbbm{1}_{G_m} = [N:G_m]=[G_{m+1}:G_m]/[G_{m+1}:N]=\frac{r(m+1)}{r'}.
\]
On the other hand, if $a_k=0$, then $b_k=1$ by Lemma \ref{lem:UVW-facts}, thus there are $r+1-r'$ representations $W_k$ which contain some $V_j$. By the same lemma there must be $r+1-r'$ of the $V_j$ whose restriction contains $\mathbbm{1}_{G_m}$ and each of these $V_j$'s must contain some $U_i$, while only one of them can contain $U_1$.  This implies that 
\[
\dim \Ind_{G_m}^{N} \mathbbm{1}_{G_m} \geq 1 + m + (r-r')(m+1).
\]
Thus we have $\frac{r(m+1)}{r'} \geq 1+m+(r-r')(m+1)$, and so
\[
rm+r \geq (r'+rr'-(r')^2)m +(r'+rr'-(r')^2)
\]
This forces $r \geq r'+rr'-(r')^2$, that is:
\[
(r'-1)(r-r') \leq 0.
\]
Therefore $r'$ must be equal to 1 or $r$.  

When $r=r'$, we have $[N:G_m]=(m+1)$, this forces there to be some $V_j$ which restricts exactly to $\mathbbm{1}_{G_m} \oplus U_1$.  Then by Clifford's Theorem, each of the $V_j$ must have dimension $m$, so the branching rules are as depicted in Figure \ref{fig:two-UVW-possibilities} for $r=3$, and extend in the obvious way to larger $r$.

\begin{lemma} \label{lem:not-fib-case}
$r'=1$.
\end{lemma}
\begin{proof}[Proof of Lemma \ref{lem:not-fib-case}]
Suppose $r>1$ and $r'=r$.  Let $V_1$ be the unique $V_j$ whose restriction contains $U_1$.  The representation $U_1$ is the $(m-1)$-dimensional representation of $G_m=(\Z/r \Z) \wr S_m$ obtained by projecting onto the symmetric group $S_m$ and applying the usual irreducible reflection representation of $S_m$.  Thus $\ker(U_1) \cong (\Z/r \Z)^m$ is the \textit{base group} (see \cite{James1981}, Chapter 4) of the wreath product.  In particular $|\ker(U_1)| = r^m$, and since $\Res V_1 = \mathbbm{1}_{G_m} \oplus U_1$, we see that $|\ker(V_1)| \geq r^m$.  By Clifford's Theorem, the representation $V_2$ is conjugate to $V_1$; that is, if $\ph_1, \ph_2: N \to GL_m$ are the corresponding maps, then for some $g \in G_{m+1}$ we have $\ph_2(x)=\ph_1(gxg^{-1})$ for all $x$.  This implies that $\ker(V_2)=g^{-1}(\ker(V_1))g$, so $|\ker(V_2)| \geq r^m$.  However $\Res V_2 = U_2$ is faithful, so $\ker(V_2)$ intersects $G_m$ trivially.  This means that the product group $K=(\ker(V_2)) \cdot G_m$ has order at least $r^m \cdot r^m m!$.  But if $r > 1$ this is greater than $|N|=r^m (m+1)!$, a contradiction.
\end{proof}

\begin{lemma} \label{lem:faithfulness}
We have $N=G_{m+1}$ (and so we can identify $V_j=W_j$ for all $j$,  relabeling if necessary).  Let $V_1$ be the unique $V_j$ whose restriction contains $U_1$, then $\dim(V_1)=m$ and $\dim(V_j)=m+1$ for $j \neq 1$.  Futhermore, $V_j$ is a faithful reflection representation of $G_{m+1}$ for $j \neq 1$ (if $r=1$ then $V_1$ is a faithful reflection representation). 
\end{lemma}
\begin{proof}[Proof of Lemma \ref{lem:faithfulness}]
By definition $r'=[G_{m+1}:N]$, since $r'=1$ by Lemma \ref{lem:not-fib-case} we have $G_{m+1}=N$.  The $V_j$, which were defined by restricting the $W_k$ to $N$ are thus the same as the $W_k$ and we can relabel to let $V_j=W_j$.  The representation $U_1$ has dimension $m-1$ and each other $U_i$ has dimension $m$.  Letting $U_i$ be the unique nontrivial representation appearing in $\Res V_i$ for all $i$ we have $\Res^{G_{m+1}}_{G_m} V_i=U_i \oplus \mathbbm{1}_{G_m}$, so $\dim(V_1)=m$ and $\dim(V_j)=m+1$ for $j \neq 1$.  

The representations $U_i$ for $i \neq 1$ can be realized as permutation matrices but with the 1's replaced by character values of one of the nontrivial representations of $\Z/r \Z$ (see \cite{James1981}).  Let $r$ be prime, then these representations of $\Z / r \Z$ are faithful, and thus $U_2,...,U_r$ are faithful as well.  Let $\psi_2: G_m \to GL(U_2)$ and $\ph_2: G_{m+1} \to GL(V_2)$ be the maps corresponding to the representations $U_2,V_2$, where, since $\Res V_2 = U_2 \oplus \mathbbm{1}_{G_m}$, we identify $GL(U_2)$ as a subgroup of $GL(V_2)$ in the natural way.  The representation $U_2$ is a reflection representation of $G_m$, so $G_m$ is generated by elements $X_m \subset G_m$ which act in $U_2$ as complex reflections.  Now, since $G_{m+1}=N$ is the smallest normal subgroup containing $G_m$, it is generated by the conjugacy classes in $G_{m+1}$ of the elements in $X_m$.  Since $\Res V_2 = U_2 \oplus \mathbbm{1}_{G_m}$, the elements of $X_m$ also act as complex reflections in $V_2$, and thus so do their $G_{m+1}$ conjugates.  Therefore $\im(\ph_2)$ is generated as a group by complex reflections.  It remains to check that $\ph_2$ is faithful.

Let $H_{m+1}:=\im(\ph_2)$ and let $G_m \cong H_m := \im(\psi_2) \subset H_{m+1}$.  We know that $\psi_2$ is faithful, so $|H_m|=r^mm!$, and that $H_{m+1}$ is an irreducible complex reflection group of rank $m+1$ whose order divides $|G_{m+1}|=r^{m+1}(m+1)!$.  By the classification of irreducible complex reflection groups \cite{Shephard1954}, we must either have $|H_{m+1}|=r^m(m+1)!$ or $r^{m+1}(m+1)!$ (all of the exceptional reflection groups $H$ have multiple prime factors dividing $|H|/\rank(H)!$, thus $H_{m+1}$ must belong to the infinite family).  We have $[H_{m+1}:H_m]=\dim(\Ind_{H_m}^{H_{m+1}} \mathbbm{1}_{H_m})$.  But $\Ind_{H_m}^{H_{m+1}} \mathbbm{1}_{H_m}$ contains summands $\mathbbm{1}_{H_{m+1}}$ and $V_2$, so $[H_{m+1}:H_m] \geq m+2$.  This forces $|H_{m+1}|=r^{m+1}(m+1)!$, so $\ph_2$ is faithful.  Therefore, as it is an irreducible complex reflection group of order $r^{m+1}(m+1)!$ and rank $m+1$, we must have $G_{m+1} \cong (\Z/r \Z) \wr S_{m+1}$.

Finally, if $r=1$, then $U_1$ is faithful, it is the usual reflection representation of $S_m$.  We see that $\Ind \mathbbm{1}_{G_m} = V_1 \oplus \mathbbm{1}_{G_{m+1}}$.  Since $U_1$ is faithful and appears in the restriction of $V_1$, we see that $V_1$ is faithful on any conjugacy class intersecting $G_m$.  On any conjugacy class not intersecting $G_m$, the character value of $\Ind \mathbbm{1}_{G_m}$ is 0, and so $\chi_{V_1}=-1$ on these conjugacy classes.  Thus $V_1$ is faithful and $G_{m+1}$ is an irreducible complex reflection group of rank $m$ and order $(m+1)!$.  By the classification, this forces $G_{m+1} \cong S_{m+1}$.
\end{proof}


By \cite{Shephard1954}, we can take $H_{m+1}$ to be the standard realization of $(\Z/r \Z) \wr S_{m+1}$ in $GL(V_2)$ as monomial matrices.  Since $\Res^{G_{m+1}}_{G_m} V_2  \cong \mathbbm{1}_{G_m} \oplus U_2$, we see that $H_m$ is conjugate in $H_{m+1}$ to the standard embedding of $m \times m$ monomial matrices into $H_{m+1}$.  This shows that the embedding $G_m \subset G_{m+1}$ is conjugate to the usual one, and so we have the usual branching rules, thus $P_{[0,m+1]} \cong (Y^r)_{[0,m+1]}$.  The proof of Proposition \ref{prop:inductive-step} is now complete.  This result, together with the base case Proposition \ref{prop:base-case}, implies Theorem \ref{thm:main-theorem}.
\end{proof}

\section*{Acknowledgements}
I am grateful to Christopher Ryba for his thorough answers to many of my questions about the symmetric group.

\bibliographystyle{plain}

\end{document}